\documentclass{amsart}

\usepackage{enumitem}
\usepackage{amssymb,amsmath,amsthm}
\usepackage{caption}
\usepackage{float}
\usepackage{graphicx}
\usepackage{subcaption}

\newtheorem{theorem}{Theorem}[section]

\newtheorem{definition}[theorem]{Definition}

\begin{document}
\title{Cartan Horadam Spinors}

%    Remove any unused author tags.
%    author one information
\author{Selime Beyza ÖZÇEVİK}
\address{Department of Statistics, Yüksek İhtisas University, Ankara, 06520,Turkey}
\curraddr{}
\email{ozcevikbeyza8@gmail.com}
\thanks{}
%    author two information
\author{Abdullah DERTLİ}
\address{Department of Mathematics, Ondokuz Mayıs University, Samsun, 55139,Turkey}
\curraddr{}
\email{abdullah.dertli@gmail.com}
\thanks{}
.

.

\date{}

\dedicatory{}

\begin{abstract}
Number sequences with wide-ranging applications in mathematics, physics, medicine, and engineering remain an active research topic. This study examines these sequences through the general framework of Horadam numbers and their special cases associated with Cartan numbers. By defining spinor transformations on the resulting structures, new types of spinors are introduced and their key properties are analyzed. The proposed approach bridges distinct yet contemporary research areas, contributing to a broader interdisciplinary perspective.

\end{abstract}

\maketitle

\section{Introduction}
Number sequences are fundamental structures that appear not only in number theory but also in computer science, cryptography, engineering, and even in nature, such as in DNA helices and spiral galaxies. Since they are based on modeling the relationships between numbers, they remain a current and active area of research. Historically, Fibonacci sequences were defined in the 13th century, but from the mid-20th century onwards, studies have been conducted on special number sequences such as Lucas, Pell, and Jacobsthal. The Horadam sequence, on the other hand, represents the most general form of all these special sequences. Therefore, depending on their initial conditions and coefficients, all number sequences that take specific names are, in fact, special cases of the Horadam sequence.

The $n$th Horadam number, denoted by $H_n$, is defined for $n \geq 1$ by the following recurrence relation, where $a$, $b$, $p$, and $q$ are integers:

\[
H_n = pH_{n-1} + qH_{n-2}, \quad n \geq 2
\]
with the initial conditions

\[
H_0 = a, \quad H_1 = b.
\]

The characteristic equation of this recurrence relation is given by

\[
x^2 - px - q = 0,
\]
whose roots are

\[
x_{1,2} = \frac{p \pm \sqrt{p^2 + 4q}}{2}.
\]

Depending on the values of $p$, $q$, $a$, and $b$, various numerical sequences can be derived, some examples of which are presented below \cite{koshy2019fibonacci}.

\begin{center}
\begin{tabular}{lllll}
$p$ & $q$ & $a$ & $b$ & Sequence \\ \hline
1 & 1 & 0 & 1 & Fibonacci number sequence \\
1 & 1 & 2 & 1 & Lucas number sequence \\
2 & 1 & 0 & 1 & Pell number sequence \\
2 & 1 & 2 & 1 & Pell Lucas number sequence \\
1 & 2 & 0 & 1 & Jacobsthal number sequence \\
1 & 2 & 2 & 1 & Jacobsthal Lucas number sequence \\
\end{tabular}
\end{center}

Quaternions, defined in 1843 by William Rowan Hamilton, are fundamental structures used in mathematics and physics to represent spatial transformations\cite{hamilton1866elements}. In three-dimensional space, quaternions provide an algebraic representation of rotation operations and are widely applied today in various fields such as robotics, signal processing, and cryptography \cite{cooke1992npsnet,demirci2020geometric,frenkel2008quaternionic,de1999quaternionic,vince2011quaternions}.
A real quaternion is defined with $q=q_0+\mathbf{i} q_1+\mathbf{j} q_2+\mathbf{k} q_3$, where $q_0, q_1, q_2$, $q_3 \in \mathbb{R}$ and the quaternion basis $\{\mathbf{1}, \mathbf{i}, \mathbf{j}, \mathbf{k}\}$ is given such that
$$
\mathbf{i}^2=\mathbf{j}^2=\mathbf{k}^2=-1, \quad \mathrm{ij}=-\mathrm{ji}=\mathbf{k}, \quad \mathbf{j k}=-\mathrm{kj}=\mathrm{i}, \quad \mathbf{k i}=-\mathrm{ik}=\mathbf{j} .
$$
Let $q_0=S_q$ and $\mathbf{V}_q=\mathbf{i} q_1+\mathbf{j} q_2+\mathbf{k} q_3$ be scalar and vectorial parts of the quaternion $q$. So, we can write the quaternion $q$ as $q=S_q+\mathbf{V}_q$. The set of these quaternions is $\mathbb{K}$. Let $p=S_p+\mathbf{V}_p, q=S_q+\mathbf{V}_q \in \mathbb{K}$ be two real quaternions. So, the quaternion product of these quaternions
$$
p \times q=S_p S_q-\left\langle\mathbf{V}_p, \mathbf{V}_q\right\rangle+S_p \mathbf{V}_q+S_q \mathbf{V}_p+\mathbf{V}_p \wedge \mathbf{V}_q,
$$
where $\langle$,$\rangle$ is usual inner product and $\wedge$ is cross product in real vector space $\mathbb{R}^3$. The product of two real quaternions is non-commutative. Moreover, if we consider that $q^*$ is the conjugate of the quaternion $q, q^*$ is equal to $q^*=S_q-\boldsymbol{V}_q$. In addition, the norm of a quaternion $q$ is
$$
N(q)=\sqrt{q_1^2+q_2^2+q_3^2+q_4^2}\  .
$$
If $N(q)=1$, then $q$ is called a unit quaternion \cite{hamilton1866elements}.

In the literature, there exist numerous studies in which quaternions—widely used in both mathematics and physics—are examined together with various number sequences. Patel and Ray introduced the $(p,q)$-Fibonacci and $(p,q)$-Lucas quaternions by extending the classical Fibonacci and Lucas sequences \cite{patel2019properties}. Tan and Leung, on the other hand, defined the Horadam quaternion sequence and presented several results that establish relationships between the Horadam, $(p,q)$-Fibonacci, and $(p,q)$-Lucas quaternions \cite{tan2020some}. Çimen and İpek, in their study, examined the Pell and Pell–Lucas quaternions in more detail and obtained various identities \cite{ccimen2016pell}.In his study, Taşcı generalized the classical Jacobsthal and Jacobsthal–Lucas number sequences into quaternion forms, defined their recurrence relations, and derived fundamental properties such as the Binet formula and the generating function \cite{tasci2017k}.

Spinors, being a more abstract concept compared to quaternions, are mathematical entities that describe the rotations of particles in special relativity and quantum mechanics \cite{vivarelli1984development}.
Let us consider the vector $\left( {{\alpha _1},{\alpha _2},{\alpha _3}} \right)$ with ${\alpha _1}^2 + {\alpha _2}^2 + {\alpha _3}^2 = 0$ in the complex vector space $\mathbb{C}^3$. These vectors form a two-dimensional surface in the two-dimensional  $\mathbb{C}^2$subspace of  $\mathbb{C}^3$  .
If the parameters of this two-dimensional surface are taken as ${\varphi _1}$ and ${\varphi _2}$ , the following equations can be  written

\[\begin{array}{l}
{\alpha _1} = {\varphi _1}^2 - {\varphi _2}^2\\
{\alpha _2} = i\left( {{\varphi _1}^2 + {\varphi _2}^2} \right)\\
{\alpha _3} =  - 2{\varphi _1}{\varphi _2}
\end{array}\]
Thus, each isotropic vector in $\mathbb{C}^3$ corresponds to a vector in $\mathbb{C}^2$ and vice versa. The vector$\varphi  = \left( {{\varphi _1},{\varphi _2}} \right) \cong \left[ {\begin{array}{*{20}{c}}
{{\varphi _1}}\\
{{\varphi _2}}
\end{array}} \right]$obtained in this way is called a spinor. The set of spinors is denoted by  $\mathbb{S}$ \cite{cartan2012theory}. For the spinor $\varphi$, by using the complex conjugate $\overline \varphi $ and the matrix $C = \left[ {\begin{array}{*{20}{c}}
0&1\\
{ - 1}&0
\end{array}} \right]$,  Cartan is defined as spinor conjugate, and Castillo and Barrales are defined  spinor mate as follows, respectively \cite{cartan2012theory, torres2004spinor}.
\begin{itemize}
\item $\widetilde \varphi  = iC\overline \varphi,$
\item$\widehat \varphi  =  - C\overline \varphi. $
\end{itemize}
 In his study, Vivarelli defined the relationship between quaternions and spinors as follows \cite{vivarelli1984development}.
This correspondence establishes a mapping from any quaternion \( q = q_0 + i q_1 + j q_2 + k q_3 \) to a spinor \( \psi = \begin{pmatrix} \psi_1 \\ \psi_2 \end{pmatrix} \), which is given by:

$$
f: K \to S
\quad q \mapsto f(q_0 + i q_1 + j q_2 + k q_3) =
\left[\begin{array}{ll}
q_3 + iq_0 \\
q_1 +iq_2
\end{array}\right] . \\
$$ 

This formulation provides a spinor representation of the kinematics of rotation, which generalizes the quaternion formulation. Additionally, it is evident that this function is both linear and injective.

Cartan numbers are defined by Öztürk in 2023. The set of these numbers denoted by S is defined as follows \cite{ozturk2023introduction}.
$$
S=\left\{a+bi+cj+dk \mid a,b,c,d \in \mathbb{R}, i^2=1, j^2=k^2=0, jk=1+i, kj=1-i\right\}.
$$
Then, $1,i,j,k$ are called Cartan units. The Cartan number system has been developed based on the following matrix representations.
$$
1\leftrightarrow 
\begin{pmatrix}
 1 & 0 \\
 0 & 1 
\end{pmatrix},
 i\leftrightarrow 
\begin{pmatrix}
 0 & 1 \\
 1 & 0 
\end{pmatrix},
 j\leftrightarrow 
\begin{pmatrix}
 1 & -1 \\
 1 & -1 
\end{pmatrix},
 k\leftrightarrow \frac{1}{2}
\begin{pmatrix}
 1 & 1 \\
 -1 & -1 
\end{pmatrix}.
$$
For the Cartan number $q=q_{1}+q_{2}i+q_{3}j+q_{4}k$, $q_{1}$ is called the scalar part and $q_{2}i+q_{3}j+q_{4}k$ is called the vector part.

The sum of two Cartan numbers is defined by summing their components. The addition operation in the Cartan numbers is both commutative and associative. Zero is a null element and the inverse element of $q$ is $-q$, which is defined as having all components of $q$ change in their signs. As a result, $(S, +)$ is an abelian group.
The Cartan numbers product
$$
\left(q_{1}+q_{2}i+q_{3}j+q_{4}k\right) \left(p_{1}+p_{2}i+p_{3}j+p_{4}k\right)
$$
is obtained by distributing the terms on the right. Cartan numbers are not commutative, but they have the property of associativity.

\[
\begin{array}{c|cccc}
\cdot & 1 & i & j & k \\ \hline
1 & 1 & i & j & k \\
i & i & 1 & j & -k \\
j & j & -j & 0 & 1+i \\
k & k & k & 1-i & 0 \\
\end{array}
\]

The conjugate of a Cartan number$q=q_{1}+q_{2}i+q_{3}j+q_{4}k$, $q_{1}$ , denoted by 

$$
\overline{q}=q_{1}-q_{2}i-q_{3}j-q_{4}k.
$$ 

The conjugate of the sum of Cartan numbers is as follows
$$
\overline{q+p}=\overline{q}+\overline{p} .
$$
Moreover, we have $q\overline{q}=\overline{q}q$ from the Cartan numbers product. As a result of this product
$$
C\left(q\right)=q\overline{q}=\overline{q}q={q_{1}}^2-{q_{2}}^2-2q_{3}q_{4}
$$
is called the character of the Cartan number $q$.

Özçevik et al., in their study, applied a spinor transformation to the Cartan numbers and defined the Cartan spinor correspondence to any Cartan number $q=a+b i+c j+d k$ is defined by $\mathcal{E}$ transformation as follows \cite{ozcevik2025cartan}.

$$
\begin{aligned}
\varepsilon: \mathbb{C} & \longrightarrow S \\
q & \longrightarrow \mathcal{}(a+b i+c j+d k)=\left[\begin{array}{l}
a+\left(c+\frac{d}{2}\right) i \\
\left(c-\frac{d}{2}\right)+b i
\end{array}\right] \equiv S C \ .
\end{aligned}
$$

In this study, the Horadam numbers are examined together with the set of Cartan numbers, and some fundamental properties of the Horadam sequence are reinterpreted by combining them with the Cartan number algebra. The resulting structure offers a new perspective obtained within a different algebraic framework. Furthermore, the derived results will be classified based on the coefficients and initial conditions of other special integer sequences.

\section{Main Theorems and Proofs}

\begin{definition} 
Let $CH$ is be the set of Cartan Horadam number. Then we defined this set as follows.
\end{definition}
\[
CH=\left\{CW_n=w_n+w_{n+1}i+w_{n+2}j+w_{n+3}k \mid  i^2=1, j^2=k^2=0, jk=1+i, kj=1-i\right\}
\]
where $w_n$ is $n$th Horadam number. \\

\begin{tabular}{|c|c|c|c|c|}
\hline
\textbf{a} & \textbf{b} & \textbf{p} & \textbf{q} & \textbf{Type of Cartan sequence} \\ \hline
0 & 1 & 1 & 1 & Cartan Fibonacci sequence \\ \hline
2 & 1 & 1 & 1 & Cartan Lucas sequence \\ \hline
0 & 1 & 2 & 1 & Cartan Pell sequence \\ \hline
0 & 1 & 1 & 2 & Cartan Jacobsthal sequence \\ \hline
2 & 1 & 2 & 1 & Cartan Pell Lucas sequence \\ \hline
2 & 1 & 1 & 2 & Cartan Jacobsthal Lucas sequence \\ \hline
\end{tabular} 
\\ \\

Therefore, by altering the coefficients and initial conditions corresponding to the defined Cartan–Horadam number, other types of Cartan numbers can be successively constructed as follows. 

\begin{definition}
The \textit{Cartan Pell number} is defined as follows.
\end{definition}
Let $CP_n$ denote the $n$th Cartan Pell number. For $n\geq0$,

\[
CP_n=P_{n}+P_{n+1}i+P_{n+2}j+P_{n+3}k
\]
where, $P_n$ is $n$th Pell number.

For example,
\[
\begin{aligned}
& CP_{0} = i + 2j + 5k, \\
& CP_{1} = 1 + 2i + 5j + 12k, \\
& CP_{2} = 2 + 5i + 12j + 29k, \\
& CP_{3} = 5 + 12i + 29j + 70k
\end{aligned}
\]
can be given.

\begin{definition}
Let $Cp_n$ denote the $n$th Cartan Pell Lucas number. For $n \geq 0$,
\[
Cp_n = p_{n} + p_{n+1}i + p_{n+2}j + p_{n+3}k
\]
where $p_n$ is the $n$th Pell Lucas number.
\end{definition}

For example,
\[
\begin{aligned}
& Cp_{0} = 2 + i + 4j + 9k, \\
& Cp_{1} = 1 + 4i + 9j + 22k, \\
& Cp_{2} = 4 + 9i + 22j + 53k, \\
& Cp_{3} = 9 + 22i + 53j + 128k
\end{aligned}
\]
can be given.

\begin{theorem}
For $n \geq 0$, we have
\begin{itemize}
\item $CP_{n} + CP_{n+1} = Cp_{n+1}$, 
\item $CP_{n+1} - CP_{n} = Cp_{n}$,
\item $CP_{n-1} + CP_{n+1} = Cp_{n}$,
\item $2CP_{n} + Cp_{n} = Cp_{n+1}$.
\end{itemize}
\end{theorem}

\begin{proof}
The $n$th and $(n+1)$th Cartan Pell numbers are as follows:
\begin{align*}
CP_n &= P_{n} + P_{n+1}i + P_{n+2}j + P_{n+3}k, \\
CP_{n+1} &= P_{n+1} + P_{n+2}i + P_{n+3}j + P_{n+4}k.
\end{align*}

The sum of these numbers is obtained as
\begin{align*}
CP_n + CP_{n+1} &= (P_{n} + P_{n+1}) 
+ (P_{n+1} + P_{n+2})i 
+ (P_{n+2} + P_{n+3})j 
+ (P_{n+3} + P_{n+4})k.
\end{align*}

Based on the relation between the Pell and Pell–Lucas numbers,
\[
p_{n+1} = P_{n} + P_{n+1},
\]
we obtain
\begin{align*}
CP_n + CP_{n+1} 
&= p_{n+1} + p_{n+2}i + p_{n+3}j + p_{n+4}k \\
&= Cp_{n}.
\end{align*}

The others can be easily shown in a similar way, based on the relations between the Pell and Pell–Lucas numbers.
\end{proof}

\begin{definition}
Let $CJ_n$ be the $n$th Cartan Jacobsthal number. For $n \geq 0$,
\[
CJ_n = J_{n} + J_{n+1}i + J_{n+2}j + J_{n+3}k
\]
where $J_n$ is the $n$th Jacobsthal number.
\end{definition}

For example,
\[
\begin{aligned}
& CJ_{0} = i + j + 3k, \\
& CJ_{1} = 1 + i + 3j + 5k, \\
& CJ_{2} = 1 + 3i + 5j + 11k, \\
& CJ_{3} = 3 + 5i + 11j + 21k
\end{aligned}
\]
can be given.

\begin{definition}
Let $Cj_n$ be the $n$th Cartan Jacobsthal Lucas number. For $n \geq 0$,
\[
Cj_n = j_{n} + j_{n+1}i + j_{n+2}j + j_{n+3}k
\]
where $j_n$ is the $n$th Jacobsthal Lucas number.
\end{definition}

For example,
\[
\begin{aligned}
& Cj_{0} = 2 + i + 5j + 7k, \\
& Cj_{1} = 1 + 5i + 7j + 17k, \\
& Cj_{2} = 5 + 7i + 17j + 31k, \\
& Cj_{3} = 7 + 31i + 65j + 127k
\end{aligned}
\]
can be given.

\begin{theorem}
For $n \geq 1$, we have
\begin{itemize}
\item $CJ_{n} + Cj_{n} =2 CJ_{n}$, 
\item $3CJ_{n+1} + Cj_{n} =2^{n+1}(1+2i+4j+8k)$,
\item $Cj_{n+1} +2 Cj_{n-1} = 9CJ_{n}$.
\end{itemize}
\end{theorem}

\begin{proof}
The $n$th Cartan Jacobsthal and $n$th Cartan Jacobsthal Lucas numbers are as follows, respectively.
\begin{align*}
CJ_n &= J_{n} + J_{n+1}i + J_{n+2}j + J_{n+3}k, \\
Cj_{n} &= j_{n} + j_{n+1}i + j_{n+2}j + j_{n+3}k.
\end{align*}

The sum of these numbers is obtained as
\begin{align*}
CJ_n + Cj_{n} &= (J_{n} + j_{n}) 
+ (J_{n+1} + j_{n+1})i 
+ (J_{n+2} + j_{n+2})j 
+ (J_{n+3} + j_{n+3})k.
\end{align*}

Based on the relation between the Jacobsthal and Jacobsthal–Lucas numbers,
\[
j_{n}+ J_{n} =2 J_{n},
\]
we obtain
\begin{align*}
CJ_n + Cj_{n} 
&= 2(J_{n} + J_{n+1}i + J_{n+2}j + J_{n+3}k \\
&= 2CJ_{n}.
\end{align*}
The others can be easily shown in a similar way.
\end{proof}

\begin{definition}
The conjugate of the Cartan Horadam number is defined as follows:
\[
CW_n = w_n - w_{n+1} i - w_{n+2} j - w_{n+3} k \ .
\]
\end{definition}
Similarly, we define the conjugates of other types of Cartan numbers as follows.

\textit{The conjugate of the Cartan Pell number is}
\[
\overline{CP_n} = P_n - P_{n+1} i - P_{n+2} j - P_{n+3} k \ .
\]

\textit{The conjugate of the Cartan Jacobsthal number is}
\[
\overline{CJ_n} = J_n - J_{n+1} i - J_{n+2} j - J_{n+3} k \ .
\]

\textit{The conjugate of the Cartan Pell–Lucas number is}
\[
\overline{Cp_n} = p_n - p_{n+1} i - p_{n+2} j - p_{n+3} k \ .
\]

\textit{The conjugate of the Cartan Jacobsthal–Lucas number is}
\[
\overline{Cj_n} = j_n - j_{n+1} i - j_{n+2} j - j_{n+3} k \ .
\]

\begin{definition}
Since the ring of Cartan numbers is isomorphic to the ring of $2 \times 2$ matrices,
the matrix representation of the Cartan Horadam number is defined as follows via the isomorphism $\theta$.
\end{definition}

$$
\theta: CH \to M_{(2\times2)}
\quad CW_{n} \mapsto \theta(w_0 + i w_1 + j w_2 + k w_3) =
\left[\begin{array}{ll}
w_{n}+w_{n+2}+\frac{w_{n+3}}{2}  & w_{n+1}-w_{n+2}+\frac{w_{n+3}}{2}  \\
w_{n+1}+w_{n+2}-\frac{w_{n+3}}{2} & w_{n}-w_{n+2}-\frac{w_{n+3}}{2}  \\
\end{array}\right] . \\
$$

Similarly, by using the same isomorphism, the matrix representations of the Cartan numbers corresponding 
to other special number types are obtained respectively as follows.

\textit{The matrix representation of the Cartan Pell number is}

$$
\theta: CP \to M_{(2\times2)}
\quad CP_{n} \mapsto \theta(P_0 + i P_1 + j P_2 + k P_3) =
\left[\begin{array}{ll}
P_{n}+P_{n+2}+\frac{P_{n+3}}{2} & P_{n+1}-P_{n+2}+\frac{P_{n+3}}{2}  \\
P_{n+1}+P_{n+2}-\frac{P_{n+3}}{2}  & P_{n}-P_{n+2}-\frac{P_{n+3}}{2}  \\
\end{array}\right] . \\
$$ 

\textit{The matrix representation of the Cartan Jacobsthal number is}

$$
\theta: CJ \to M_{(2\times2)}
\quad CJ_{n} \mapsto \theta(J_0 + i J_1 + j J_2 + k J_3) =
\left[\begin{array}{ll}
J_{n}+J_{n+2}+\frac{J_{n+3}}{2} & J_{n+1}-J_{n+2}+\frac{J_{n+3}}{2}  \\
J_{n+1}+J_{n+2}-\frac{J_{n+3}}{2}  & J_{n}-J_{n+2}-\frac{J_{n+3}}{2}  \\
\end{array}\right] . \\
$$ 

\textit{The matrix representation of the Cartan Pell Lucas number is}

$$
\theta: Cp \to M_{(2\times2)}
\quad Cp_{n} \mapsto \theta(p_0 + i p_1 + j p_2 + k p_3) =
\left[\begin{array}{ll}
p_{n}+p_{n+2}+\frac{p_{n+3}}{2} & p_{n+1}-p_{n+2}+\frac{p_{n+3}}{2}  \\
p_{n+1}+p_{n+2}-\frac{p_{n+3}}{2}  & p_{n}-p_{n+2}-\frac{p_{n+3}}{2}  \\
\end{array}\right] . \\
$$ 

\textit{The matrix representation of the Cartan Jacobsthal Lucas number is}

$$
\theta: Cj \to M_{(2\times2)}
\quad Cj_{n} \mapsto \theta(j_0 + i j_1 + j j_2 + k j_3) =
\left[\begin{array}{ll}
j_{n}+j_{n+2}+\frac{j_{n+3}}{2}  & j_{n+1}-j_{n+2}+\frac{j_{n+3}}{2}  \\
j_{n+1}+j_{n+2}-\frac{j_{n+3}}{2} & j_{n}-j_{n+2}-\frac{j_{n+3}}{2}  \\
\end{array}\right] . \\
$$ 

\begin{definition}
For $n \geq 1$, the Cartan Horadam number sequence is defined by this recurrence relation
$$
CW_{n+1}=pCW_{n}+qCW_{n-1}
$$
with initial conditions, 
$$
\begin{aligned}
&CW_{0}=a+bi+(pb+qa)j+(p^{2}b+pqa+pb)k, \\
&CW_{1}=b+(pb+qa)i+(p^{2}b+pqa+pb)j+(p^{3}b+p^{2}qa+2pbqa+q^{2}a)k.
\end{aligned} .
$$

\end{definition}

Similarly, we obtain other Cartan number sequences by altering the coefficients of recurrence relations and initial conditions of the Cartan Horadam sequence as follows.

For $n \geq 1$, \textit{the Cartan Pell sequence} is  defined by this recurrence relation
$$
CP_{n+1}=2CP_{n}+CP_{n-1}
$$
with initial conditions, 
$$
\begin{aligned}
&CP_{0}=i+2j+5k, \\
&CP_{1}=1+2i+5j+12k .
\end{aligned} 
$$

For $n \geq 1$, \textit{the Cartan Jacobsthal sequence} is  defined by this recurrence relation
$$
CJ_{n+1}=CJ_{n}+2CJ_{n-1}
$$
with initial conditions, 
$$
\begin{aligned}
&CJ_{0}=i+j+3k, \\
&CJ_{1}=1+i+3j+5k.
\end{aligned}
$$

For $n \geq 1$, \textit{the Cartan Pell Lucas sequence} is  defined by this recurrence relation
$$
Cp_{n+1}=2Cp_{n}+Cp_{n-1}
$$
with initial conditions, 
$$
\begin{aligned}
&Cp_{0}=2+i+4j+9k, \\
&Cp_{1}=1+4i+9j+22k.
\end{aligned}
$$

For $n \geq 1$, \textit{the Cartan Jacobsthal Lucas sequence} is  defined by this recurrence relation
$$
Cj_{n+1}=Cj_{n}+2Cj_{n-1}
$$

with initial conditions, 

$$
\begin{aligned}
&Cj_{0}=2+i+5j+7k, \\
&Cj_{1}=1+5i+7j+17k.
\end{aligned}
$$

\begin{theorem} 
The Binet formula for the Cartan Horadam number is as follows.
$$
CW_n=X \alpha^{n}+Y\beta^{n}
$$
where,

$$
\begin{aligned}
&X=\frac{1}{2\sqrt{d}}\left(2CH_{1}-pCH_{0}+\sqrt{d} CH_{0}\right), \\
&Y=\frac{1}{2\sqrt{d}}\left(CH_{0}\left(\sqrt{d}+p\right)-2CH_{1}\right) 
\end{aligned}
$$
and $d=p^{2}+4q$.

\end{theorem}

\textit{The Binet formula for the Cartan Pell number} is as follows.
$$
CP_n=Aa^{n}+Bb^{n}
$$
where,

$$
\begin{aligned}
&A=\frac{1}{2\sqrt{2}}\left(1+\left(\sqrt{2}+1\right)i +\left(4+2\sqrt{2}\right)+j+\left(7+5\sqrt{2}\right)k\right), \\
&B=\frac{1}{2\sqrt{2}}\left(-1+\left(\sqrt{2}-1\right)i +\left(-4+2\sqrt{2}\right)j+\left(-7+5\sqrt{2}\right)k \right).
\end{aligned}
$$

\textit{The Binet formula for the Cartan Jacobsthal number} is as follows.
$$
CJ_n=Cc^{n}+Dd^{n}
$$

where

$$
\begin{aligned}
&C=1+2j+2k, \\
&D=-1+i-j+k.
\end{aligned}
$$

\textit{The Binet formula for Cartan Pell Lucas number} is as follows.
$$
Cp_n=A^{*}a^{n}+B^{*}b^{n}
$$
where

$$
\begin{aligned}
&A^{*}=\left(1+\left(\sqrt{2}+1\right)i +\left(3+2\sqrt{2}\right)j+\left(7+5\sqrt{2}\right)k\right),\\
&B^{*}=\left(1+\left(-\sqrt{2}+1\right)i +\left(3-2\sqrt{2}\right)j+\left(7-5\sqrt{2}\right)k \right).
\end{aligned}
$$

\textit{The Binet formula for Cartan Jacobsthal Lucas number} is as follows.
$$
Cj_n=C^{*}c^{n}+D^{*}d^{n}
$$
where,

$$
\begin{aligned}
&C^{*}=-1+4i+2j+10k,\\
&D^{*}=3-3i+3j-3k.
\end{aligned}
$$

\begin{theorem} The generating function for the Cartan Horadam sequence is obtained as follows:
$$
C_W(x)=\frac{1}{1- px- qx^2} x\left(C F_1 - C F_0\right)+ C F_0 \text {. }
$$
\end{theorem}

\begin{proof} Assume that $CW_ n$ is the nth Cartan Horadam number and the generating function of the Cartan Horadam number sequence is
$$
c_W(x)=\sum_{n=0}^{\infty} CW_n x^n .
$$

First, the function can be written from the recurrence relation of Cartan Horadam sequence as follows:
$$
\begin{aligned}
& \sum_{n=0}^{\infty} C W_{n+2} x^n=p \sum_{n=0}^{\infty} C W_{n+1} x^n+q \sum_{n=0}^{\infty} C W_n x^n, \\
& \sum_{n=2}^{\infty} C W_n x^{n-2}=p \sum_{n=1}^{\infty} C W_n x^{n-1}+q \sum_{n=0}^{\infty} C W_n x^n .
\end{aligned}
$$

Then, the following equation can be obtained
$$
\frac{1}{x^2}\left[-C W_0-C W_1+c_f(x)\right]=p \frac{1}{x}\left[-C W_0+c_f (x)\right]+q c_f(x) .
$$

Consequently, for the Cartan Horadam sequence, the generating function is obtained as follows.
$$
c_W(x)=\frac{1}{1- px-qx^2}\left( x\left(C W_1 - C W_0\right)+ C W_0 \right) \text {. }
$$
\end{proof}

\textit{The generating function for the Cartan Pell sequence} is obtained as follows:
$$
C_P(x)=\frac{1}{1- 2x- x^2} \left(i+2j+5k +x\left(1+i+2k \right) \right)\text {. }
$$

\textit{The generating function for the Cartan Jacobsthal sequence} is obtained as follows:
$$
C_J(x)=\frac{1}{1- 2 x- x^2} \left(i+j+3k +x\left(1+2j+2k \right) \right)\text {. }
$$

\textit{The generating function for the Cartan Pell Lucas sequence} is obtained as follows:
$$
C_P(x)=\frac{1}{1- 2x- x^2} \left(2+i+4j+9k +x\left(-3+2i+j+4k \right) \right)\text {. }
$$

\textit{The generating function for the Cartan Jacobsthal Lucas sequence} is obtained as follows:
$$
C_J(x)=\frac{1}{1- x-2x^2} \left(2+i+5j+7k +x\left(-1+4i+2j+10k \right) \right)\text {. }
$$

\begin{definition} The Cartan Horadam spinor is defined by $\varepsilon_f$ and $\varepsilon_l$ transformations as follows
\end{definition}
$$
\varepsilon_W\left(H_n+H_{n+1} i+H_{n+2} j+H_{n+3} k\right)=\left[\begin{array}{l}
H_n+\left(H_{n+2}+\frac{H_{n+3}}{2}\right) i \\
\left(H_{n+2}-\frac{H_{n+3}}{2}\right)+H_{n+1} i
\end{array}\right] \equiv S C W_n \ .\\
$$

\textit{ The Cartan Pell and Pell Lucas spinor} are defined by $\varepsilon_P$ and $\varepsilon_p$ transformations as follows, respectively

$$
\begin{aligned}
& \varepsilon_P\left(P_n+P_{n+1} i+P_{n+2} j+P_{n+3} k\right)=\left[\begin{array}{l}
P_n+\left(P_{n+2}+\frac{P_{n+3}}{2}\right) i \\
\left(P_{n+2}-\frac{P_{n+3}}{2}\right)+P_{n+1} i
\end{array}\right] \equiv S C P_n \ ,\\
& \varepsilon_p\left(p_n+p_{n+1} i+p_{n+2} j+p_{n+3} k\right)=\left[\begin{array}{l}
p_n+\left(p_{n+2}+\frac{p_{n+3}}{2}\right) i \\
\left(p_{n+2}-\frac{p_{n+3}}{2}\right)+p_{n+1} i
\end{array}\right] \equiv S C p_n \ .
\end{aligned}
$$

\textit{ The Cartan Jacobsthal and Jacobsthal Lucas spinor} are defined by $\varepsilon_J$ and $\varepsilon_j$ transformations as follows, respectively
$$
\begin{aligned}
& \varepsilon_J\left(J_n+J_{n+1} i+J_{n+2} j+J_{n+3} k\right)=\left[\begin{array}{l}
J_n+\left(J_{n+2}+\frac{J_{n+3}}{2}\right) i \\
\left(J_{n+2}-\frac{J_{n+3}}{2}\right)+J_{n+1} i
\end{array}\right] \equiv S C J_n \ , \\
& \varepsilon_j\left(j_n+j_{n+1} i+j_{n+2} j+j_{n+3} k\right)=\left[\begin{array}{l}
j_n+\left(j_{n+2}+\frac{j_{n+3}}{2}\right) i \\
\left(j_{n+2}-\frac{j_{n+3}}{2}\right)+j_{n+1} i
\end{array}\right] \equiv S C j_n \ .
\end{aligned}
$$

\begin{definition} The complex, unitary, square matrix $\hat{Q}$ corresponding to Cartan Horadam number is as follows.
\end{definition}

$$
\hat{Q}_W=\left[\begin{array}{ll}
H_n+\left(H_{n+2}+\frac{H_{n+3}}{2}\right) i & \left(\frac{H_{n+3}}{2}-H_{n+2}\right)+H_{n+1} i \\
\left(H_{n+2}-\frac{H_{n+3}}{2}\right)+H_{n+1} i & H_n-\left(H_{n+2}+\frac{H_{n+3}}{2}\right) i
\end{array}\right] \ .
$$

\textit{ The matrix $\hat{Q}$ corresponding to Cartan Pell and Pell Lucas numbers} are as follows, respectively

$$
\begin{aligned}
& \hat{Q}_P=\left[\begin{array}{ll}
P_n+\left(P_{n+2}+\frac{P_{n+3}}{2}\right) i & \left(\frac{P_{n+3}}{2}-P_{n+2}\right)+P_{n+1} i \\
\left(P_{n+2}-\frac{P_{n+3}}{2}\right)+P_{n+1} i & P_n-\left(P_{n+2}+\frac{P_{n+3}}{2}\right) i
\end{array}\right] \ ,\\
& \hat{Q}_p=\left[\begin{array}{ll}
p_n+\left(p_{n+2}+\frac{p_{n+3}}{2}\right) i & \left(\frac{p_{n+3}}{2}-p_{n+2}\right)+p_{n+1} i \\
\left(p_{n+2}-\frac{p_{n+3}}{2}\right)+p_{n+1} i & p_n-\left(p_{n+2}+\frac{p_{n+3}}{2}\right) i
\end{array}\right] \ .
\end{aligned}
$$

\textit{ The matrix $\hat{Q}$ corresponding to Cartan Jacobsthal and Jacobsthal Lucas numbers} are as follows, respectively

$$
\begin{aligned}
& \hat{Q}_J=\left[\begin{array}{ll}
J_n+\left(J_{n+2}+\frac{J_{n+3}}{2}\right) i & \left(\frac{J_{n+3}}{2}-J_{n+2}\right)+J_{n+1} i \\
\left(J_{n+2}-\frac{J_{n+3}}{2}\right)+J_{n+1} i & J_n-\left(J_{n+2}+\frac{J_{n+3}}{2}\right) i
\end{array}\right] \ , \\
& \hat{Q}_j=\left[\begin{array}{ll}
j_n+\left(j_{n+2}+\frac{j_{n+3}}{2}\right) i & \left(\frac{j_{n+3}}{2}-j_{n+2}\right)+j_{n+1} i \\
\left(j_{n+2}-\frac{j_{n+3}}{2}\right)+j_{n+1} i & j_n-\left(j_{n+2}+\frac{j_{n+3}}{2}\right) i
\end{array}\right] \ .
\end{aligned}
$$

We write the correspondence spinor for $\overline{C W_n}$ by the $\varepsilon_W$ transformation as follows

$$
\varepsilon_W\left(H_n-H_{n+1} i-H_{n+2} j-H_{n+3} k\right)=\left[\begin{array}{l}
H_n-\left(H_{n+2}+\frac{H_{n+3}}{2}\right) i \\
\left(-H_{n+2}+\frac{H_{n+3}}{2}\right)-H_{n+1} i
\end{array}\right] \equiv S C W_n^* \ .\\
$$

Similarly, \textit{we obtain correspondence spinors for $\overline{C P_n}$ and $\overline{C p_n}$ } by $\varepsilon_P$ and $\varepsilon_p$ transformations, respectively

$$
\begin{aligned}
& \varepsilon_P\left(P_n-P_{n+1} i-P_{n+2} j-P_{n+3} k\right)=\left[\begin{array}{l}
P_n-\left(P_{n+2}+\frac{P_{n+3}}{2}\right) i \\
\left(-P_{n+2}+\frac{P_{n+3}}{2}\right)-P_{n+1} i
\end{array}\right] \equiv S C P_n^* \ ,\\
& \varepsilon_p\left(p_n-p_{n+1} i-p_{n+2} j-p_{n+3} k\right)=\left[\begin{array}{l}
p_n-\left(p_{n+2}+\frac{p_{n+3}}{2}\right) i \\
\left(-p_{n+2}+\frac{p_{n+3}}{2}\right)-p_{n+1} i
\end{array}\right] \equiv S C p_n^* \ .
\end{aligned}
$$

In addition, \textit{we obtain correspondence spinors for $\overline{C J_n}$ and $\overline{C j_n}$ } by $\varepsilon_J$ and $\varepsilon_J$ transformations, respectively

$$
\begin{aligned}
& \varepsilon_J\left(J_n-J_{n+1} i-J_{n+2} j-J_{n+3} k\right)=\left[\begin{array}{l}
J_n-\left(J_{n+2}+\frac{J_{n+3}}{2}\right) i \\
\left(-J_{n+2}+\frac{J_{n+3}}{2}\right)-J_{n+1} i
\end{array}\right] \equiv S C J_n^* \ ,\\
& \varepsilon_j\left(j_n-j_{n+1} i-j_{n+2} j-j_{n+3} k\right)=\left[\begin{array}{l}
j_n-\left(j_{n+2}+\frac{j_{n+3}}{2}\right) i \\
\left(-j_{n+2}+\frac{j_{n+3}}{2}\right)-j_{n+1} i
\end{array}\right] \equiv S C j_n^* \ .
\end{aligned}
$$

On the other hand, we give the ordinary complex conjugate of $SCW_n$ as follows.

$$
 \overline{S C W_n}=\left[\begin{array}{l}
H_n-i\left(H_{n+2}+\frac{H_{n+3}}{2}\right) \\
\left(H_{n+2}-\frac{H_{n+3}}{2}\right)-H_{n+1} i
\end{array}\right] \ .\\
$$

Similarly, we give \textit{the ordinary complex conjugate of $SCP_n$ and $SCp_n$} are as follows

$$
\begin{aligned}
& \overline{S C P_n}=\left[\begin{array}{l}
P_n-i\left(P_{n+2}+\frac{P_{n+3}}{2}\right) \\
\left(P_{n+2}-\frac{P_{n+3}}{2}\right)-P_{n+1} i
\end{array}\right] \ , \\
& \overline{S C p_n}=\left[\begin{array}{l}
p_n-i\left(p_{n+2}+\frac{p_{n+3}}{2}\right) \\
\left(p_{n+2}-\frac{p_{n+3}}{2}\right)-p_{n+1} i
\end{array}\right] \ .
\end{aligned}
$$

Finally, we give \textit{the ordinary complex conjugate of $SCJ_n$ and $SCj_n$} are as follows

$$
\begin{aligned}
& \overline{S C J_n}=\left[\begin{array}{l}
J_n-i\left(J_{n+2}+\frac{J_{n+3}}{2}\right) \\
\left(J_{n+2}-\frac{J_{n+3}}{2}\right)-J_{n+1} i
\end{array}\right] \ ,\\
& \overline{S C j_n}=\left[\begin{array}{l}
j_n-i\left(j_{n+2}+\frac{j_{n+3}}{2}\right) \\
\left(j_{n+2}-\frac{j_{n+3}}{2}\right)-j_{n+1} i
\end{array}\right] \ .
\end{aligned}
$$

Cartan Horadam spinor conjugate

$$
\begin{aligned}
& \tilde {SCW}_n=i C S \overline{C W}_n \\
& \left(\begin{array}{cc}
0 & i \\
-i & 0
\end{array}\right)\binom{H_n-i\left(H_{n+2}+\frac{H_{n+3}}{2}\right)}{\left(H_{n+2}-\frac{H_{n+3}}{2}\right)-H_{n+1} i} \\
& =\binom{-H_{n+1}+\left(H_{n+2}-\frac{H_{n+3}}{2}\right) i}{\left(H_{n+2}+\frac{H_{n+3}}{2}\right)-H_n i} \ .
\end{aligned}
$$

Similarly, we can write that the \textit{Cartan Pell and Pell Lucas spinor conjugate} are as follows

$$
\begin{aligned}
& \tilde {SCP}_n=i C S \overline{C P}_n \\
& \left(\begin{array}{cc}
0 & i \\
-i & 0
\end{array}\right)\binom{P_n-i\left(P_{n+2}+\frac{P_{n+3}}{2}\right)}{\left(P_{n+2}-\frac{P_{n+3}}{2}\right)-P_{n+1} i} \\
& =\binom{-P_{n+1}+\left(P_{n+2}-\frac{P_{n+3}}{2}\right) i}{\left(P_{n+2}+\frac{P_{n+3}}{2}\right)-P_n i} \ ,
\end{aligned}
$$

$$
\begin{aligned}
& \tilde {SCp}_n=i C S \overline{C p}_n \\
& \left(\begin{array}{cc}
0 & i \\
-i & 0
\end{array}\right)\binom{p_n-i\left(p_{n+2}+\frac{p_{n+3}}{2}\right)}{\left(p_{n+2}-\frac{p_{n+3}}{2}\right)-p_{n+1} i} \\
& =\binom{-p_{n+1}+\left(p_{n+2}-\frac{p_{n+3}}{2}\right) i}{\left(p_{n+2}+\frac{p_{n+3}}{2}\right)-p_n i} \ .
\end{aligned}
$$

Finally, we can write \textit{Cartan Jacobsthal and Jacobsthal Lucas spinor conjugate} are as follows

$$
\begin{aligned}
& \tilde {SCJ}_n=i C S \overline{C J}_n \\
& \left(\begin{array}{cc}
0 & i \\
-i & 0
\end{array}\right)\binom{J_n-i\left(J_{n+2}+\frac{J_{n+3}}{2}\right)}{\left(J_{n+2}-\frac{J_{n+3}}{2}\right)-J_{n+1} i} \\
& =\binom{-J_{n+1}+\left(J_{n+2}-\frac{J_{n+3}}{2}\right) i}{\left(J_{n+2}+\frac{J_{n+3}}{2}\right)-J_n i} \ ,
\end{aligned}
$$

$$
\begin{aligned}
& \tilde {SCj}_n=i C S \overline{C j}_n \\
& \left(\begin{array}{cc}
0 & i \\
-i & 0
\end{array}\right)\binom{j_n-i\left(j_{n+2}+\frac{j_{n+3}}{2}\right)}{\left(j_{n+2}-\frac{j_{n+3}}{2}\right)-j_{n+1} i} \\
& =\binom{-j_{n+1}+\left(j_{n+2}-\frac{j_{n+3}}{2}\right) i}{\left(j_{n+2}+\frac{j_{n+3}}{2}\right)-j_n i} \ .
\end{aligned}
$$

Lastly, the mate of Cartan Horadam spinor is as follows

$$
\begin{gathered}
S C W_n=-C \overline{SCW_n} \\
S C W_n=\left[\begin{array}{c}
\left(-H_{n+2}+\frac{H_{n+3}}{2}\right)+H_{n+1} i \\
H_n-\left(H_{n+2}+\frac{H_{n+3}}{2}\right) i
\end{array}\right] \ .
\end{gathered}
$$

\textit{The mate of Cartan Pell and Pell Lucas spinor} are as follows, respectively.

$$
S C P_n=\left[\begin{array}{c}
\left(-P_{n+2}+\frac{P_{n+3}}{2}\right)+P_{n+1} i \\
P_n-\left(P_{n+2}+\frac{P_{n+3}}{2}\right) i
\end{array}\right],
$$

$$
S C p_n=\left[\begin{array}{c}
\left(-p_{n+2}+\frac{p_{n+3}}{2}\right)+p_{n+1} i \\
p_n-\left(p_{n+2}+\frac{p_{n+3}}{2}\right) i
\end{array}\right].
$$

\textit{The mate of Cartan Jacobsthal and Jacobsthal Lucas spinor} are as follows, respectively.

$$
S C J_n=\left[\begin{array}{c}
\left(-J_{n+2}+\frac{J_{n+3}}{2}\right)+J_{n+1} i \\
J_n-\left(J_{n+2}+\frac{J_{n+3}}{2}\right) i
\end{array}\right],
$$

$$
S C j_n=\left[\begin{array}{c}
\left(-j_{n+2}+\frac{j_{n+3}}{2}\right)+j_{n+1} i \\
j_n-\left(j_{n+2}+\frac{j_{n+3}}{2}\right) i
\end{array}\right].
$$

\begin{definition} The Cartan Horadam spinor sequence is defined by the recurrence relation 
$$
S CW_{n+1}=S CW_n+ S CW_{n-1}
$$
with initial conditions
$$
S CW_{0}=\left[\begin{array}{ll}
H_{0}+(H_{2}+\frac{H_{3}}{2})i \\
H_{2}-\frac{H_{3}}{2}+H_{1}i
\end{array}\right]
$$
and
$$
S CW_{1}=\left[\begin{array}{ll}
H_{1}+(H_{3}+\frac{H_{4}}{2})i \\
H_{3}-\frac{H_{4}}{2}+H_{2}i
\end{array}\right]
$$
where $S CW_{n}$ is $n$th Cartan Horadam spinor, for $n\geq1.$
\end{definition}

Similarly, other spinor sequences can be written as follows.

\textit{The Cartan Pell spinor sequence} is defined by the recurrence relation 
$$
S CP_{n+1}=S CP_n+ S CP_{n-1}
$$
with initial conditions
$$
S CP_{0}=\left[\begin{array}{ll}
\frac{9}{2}i \\
-\frac{1}{2}+i
\end{array}\right]
$$
and
$$
S CP_{1}=\left[\begin{array}{ll}
1+11i \\
-1+2i
\end{array}\right]
$$
where $S CP_{n}$ is $n$th Cartan Pell spinor, for $n\geq0.$

\textit{The Cartan Pell Lucas spinor sequence} is defined by the recurrence relation 
$$
S Cp_{n+1}=S Cp_n+ S Cp_{n-1}
$$
with initial conditions
$$
S Cp_{0}=\left[\begin{array}{ll}
2+\frac{17}{2}i \\
-\frac{1}{2}+i
\end{array}\right]
$$
and
$$
S Cp_{1}=\left[\begin{array}{ll}
1+20i \\
-2+4i
\end{array}\right]
$$
where $S Cp_{n}$ is $n$th Cartan Pell Lucas spinor, for $n\geq1.$

\textit{The Cartan Jacobsthal spinor sequence} is defined by the recurrence relation 
$$
S CJ_{n+1}=S CJ_n+ S CJ_{n-1}
$$
with initial conditions
$$
S CJ_{0}=\left[\begin{array}{ll}
\frac{5}{2}i \\
-\frac{1}{2}+i
\end{array}\right]
$$
and
$$
S CJ_{1}=\left[\begin{array}{ll}
1+\frac{11}{2}i \\
\frac{1}{2}+i
\end{array}\right]
$$
where $S CJ_{n}$ is $n$th Cartan Jacobsthal spinor, for $n\geq1.$

\textit{The Cartan Jacobsthal Lucas spinor sequence} is defined by the recurrence relation 
$$
S Cj_{n+1}=S Cj_n+ S Cj_{n-1}
$$
with initial conditions
$$
S Cj_{0}=\left[\begin{array}{ll}
2+\frac{17}{2}i \\
\frac{3}{2}+i
\end{array}\right]
$$
and
$$
S Cj_{1}=\left[\begin{array}{ll}
1+\frac{31}{2}i \\
-\frac{3}{2}+5i
\end{array}\right]
$$
where $S Cj_{n}$ is $n$th Cartan Jacobsthal Lucas spinor, for $n\geq1.$

\begin{theorem} The Binet formula for the Cartan Horadam spinor sequence is as follows.
\end{theorem}

$$
SCW_{n}=\frac{1}{2\sqrt{d}}X\alpha^n+Y\beta^n
$$
where
$$
\begin{gathered} 
X=\left[\begin{array}{l}
x_{0} \\
x_{1}  
\end{array}\right],
Y=\left[\begin{array}{l}
y_{0} \\
y_{1}
\end{array}\right],
\end{gathered}
$$

\[
\begin{aligned}
x_{0} &= b - pa + a\sqrt{d}
+ \frac{(pb+qa)\big(pq+\sqrt{d}(1+p^{2}+pq)\big)
+ bq\big(q+1+p\sqrt{d}\big)}{2}i, \\[4pt]
y_{0} &= -b + pa + a\sqrt{d}
+ \frac{(pb+qa)\big(-pq+\sqrt{d}(1+p^{2}+pq)\big)
+ bq\big(-q-1+p\sqrt{d}\big)}{2}i, \\[4pt]
x_{1} &= \frac{\sqrt{d}((pb+qa)(2-p)-bq+2bi)}{2}
+ \frac{(pb+qa)(2p-p^{2}-2pq+4i)+3bpq-2bpi}{2}, \\[4pt]
y_{1} &= \frac{\sqrt{d}((pb+qa)(2-p)-bq+2bi)}{2}
- \frac{(pb+qa)(2p-p^{2}-2pq+4i)+3bpq-2bpi}{2}
\end{aligned}
\]
and
$$
d = p^{2} + 4q.
$$

\begin{proof}
The characteristic equation of Cartan Horadam spinor sequence is as follows.
$$
x^{2}-px-q=0.
$$
Then, the roots of this equation are
$$
\alpha=\frac{p+\sqrt{d}}{2}, \beta=\frac{p-\sqrt{d}}{2}.
$$

We write the Binet formula of Cartan Horadam spinor sequence as follows.
$$
SCW_{n}=X\alpha^{n}+Y\beta^{n}.
$$
where

$$
\begin{gathered} 
X=\left[\begin{array}{l}
x_{0} \\
x_{1}  
\end{array}\right],
Y=\left[\begin{array}{l}
y_{0} \\
y_{1}
\end{array}\right],
\end{gathered}
$$
matrices.
When the values for $n=0$ and $n=1$ are substituted and the equations are solved simultaneously, the result is obtained.
\end{proof}

 \textit{The Binet formula for Cartan Pell spinor sequence} is as follows
$$
\begin{gathered}
SCP_{n}=\frac{1}{4\sqrt{2}}\left(\left[\begin{array}{l}
2+(9\sqrt{2}+13)i \\
-1-\sqrt{2}+(2\sqrt{2}-2)i
\end{array}\right]+
\left[\begin{array}{l}
-2+(9\sqrt{2}-13)i \\
1-\sqrt{2}+(2\sqrt{2}-2)i
\end{array}\right]\right).
\end{gathered}
$$

 \textit{The Binet formula for Cartan Pell Lucas spinor sequence} is as follows
$$
\begin{gathered}
SCp_{n}=\frac{1}{4\sqrt{2}}\left(\left[\begin{array}{l}
2+2\sqrt{2}+(23+17\sqrt{2})i \\
-3-\sqrt{2}+(6+2\sqrt{2})i
\end{array}\right]+
\left[\begin{array}{l}
2+2\sqrt{2}+(-23+17\sqrt{2})i \\
-3-\sqrt{2}+(-6+2\sqrt{2})i
\end{array}\right]\right).
\end{gathered}
$$

 \textit{The Binet formula for Cartan Jacobsthal spinor sequence} is as follows
$$
\begin{gathered}
SCJ_{n}=\frac{1}{3}\left(\left[\begin{array}{l}
1+8i \\
2i
\end{array}\right]+
\left[\begin{array}{l}
-1- \frac{1}{2}i \\
-\frac{3}{2}+i
\end{array}\right]\right).
\end{gathered}
$$

\textit{The Binet formula for Cartan Jacobsthal Lucas spinor sequence} is as follows
$$
\begin{gathered}
SCj _{n}=\frac{1}{3}\left(\left[\begin{array}{l}
1+8i \\
2i
\end{array}\right]+
\left[\begin{array}{l}
1+ \frac{1}{2}i \\
\frac{3}{2}-i
\end{array}\right]\right).
\end{gathered}
$$

\begin{theorem} The generating function for the Cartan Horadam spinor sequence is as follows:
$$
SCF_g(x)=\frac{1}{1- x- x^2} \left[\begin{array}{l}
1+x+ \left(1+\frac{7}{2}x\right)i  \\
\frac{x}{2}+\left(1+x\right)i
\end{array}\right] \ .
$$
\end{theorem}

\begin{proof}
The generating function of the Cartan Horadam spinor sequence is
$$
SCW_g(x)=\sum_{n=0}^{\infty} SCW_n x^n 
$$
for $SCW_ n$ nth Cartan Horadam spinor.
The equality obtained when the recurrence relation of the Cartan Horadam spinor sequence is written is as follows.
$$
\begin{aligned}
& \sum_{n=0}^{\infty} SC W_{n+2} x^n= p\sum_{n=0}^{\infty} SC W_{n+1} x^n+ q\sum_{n=0}^{\infty}S C W_n x^n, \\
& \sum_{n=2}^{\infty} SC W_n x^{n-2}= p\sum_{n=1}^{\infty} SC W_n x^{n-1}+ q\sum_{n=0}^{\infty} SC W_n x^n .
\end{aligned}
$$
Then, 
$$
\frac{1}{x^2}\left[-SC W_0-xSCW_1+ SCW_g(x)\right]= \frac{p}{x}\left[-SCW_0+SCW_g (x)\right]+ qSCW_g(x) \ .
$$

Consequently, the generating function of the Cartan Horadam spinor sequence  is obtained as follows
$$
SCW_g(x)=\frac{ SCW_0 \left(1-px\right)+ SC W_1}{1-px- qx^2} \text {. }
$$
\end{proof}

\textit{The generating function of the Cartan Pell spinor sequence} is as follows
$$
SCP_g(x)=\frac{1}{1- 2x- x^2} \left[\begin{array}{l}
x\left(1+2i\right)+\frac{9}{2}i  \\
-\frac{1}{2}+i
\end{array}\right].
$$

\textit{The generating function of the Cartan Pell Lucas spinor sequence} is as follows
$$
SCp_g(x)=\frac{1}{1- 2x- x^2} \left[\begin{array}{l}
-3x+2+\left(3x+\frac{17}{2}\right)i  \\
-\frac{1}{2}+\left(2x+1\right)i
\end{array}\right].
$$

\textit{The generating function of the Cartan Jacobsthal spinor sequence} is as follows
$$
SCJ_g(x)=\frac{1}{1- x- 2x^2} \left[\begin{array}{l}
x+\left(3x+\frac{5}{2}\right)i  \\
x-\frac{1}{2}+i
\end{array}\right].
$$

\textit{The generating function of the Cartan Jacobsthal Lucas spinor sequence} is as follows
$$
SCj_g(x)=\frac{1}{1- x- 2x^2} \left[\begin{array}{l}
2-x+\left(7x+\frac{17}{2}\right)i  \\
\frac{3}{2}-3x-\left(4x+1\right)i
\end{array}\right].
$$

\section{Conclusion}
In this study, Cartan Horadam numbers were defined by combining the fundamental properties of Horadam sequences with the structure of Cartan numbers. Since the Horadam sequence represents the most general form of second-order recurrence relations, modifying its coefficients and initial conditions allowed the derivation of Cartan Pell, Cartan Pell–Lucas, Cartan Jacobsthal, and Cartan Jacobsthal–Lucas numbers. Subsequently, the Cartan Horadam sequence, whose terms consist of Cartan numbers with Horadam coefficients, was constructed, and several of its fundamental properties, such as the Binet formula and the generating function, were established.
Similar procedures were applied to the special cases of the Cartan Horadam numbers. Finally, a spinor transformation was applied to these special Cartan numbers, leading to the construction of new Cartan spinors and the formation of corresponding spinor sequences.
Overall, this study unites several contemporary concepts in the literature and provides a bridge between pure mathematical structures and applied fields such as physics, quantum theory, and cryptography, serving as a potential foundation for future research.

\bibliographystyle{abbrv} 
\bibliography{ref}

\end{document}